\documentclass[12pt]{article}
\usepackage{latexsym,amsmath,theorem,graphicx,amssymb}
\usepackage[utf8]{inputenc}

\numberwithin{equation}{section}

\newtheorem{teo}{Theorem}[section]
\newtheorem{lemma}[teo]{Lemma}

{\theorembodyfont{\rmfamily}\newtheorem{re}[teo]{Remark}}
{\theorembodyfont{\rmfamily}}

\DeclareMathOperator{\EE}{\mathbb{E}}
\DeclareMathOperator{\PP}{\mathbb{P}}
 \DeclareMathOperator{\Cov}{Cov}

\newcommand{\ind}{\mathbf{1}}
\newcommand{\complex}{\mathbb{C}}
\newcommand{\real}{\mathbb{R}}

\newcommand{\integer}{\mathbb{Z}}
\renewcommand{\natural}{\mathbb{N}}

\newcommand{\as}{\emph{a.s. }}
\newcommand{\ie}{\emph{i.e. }}
\newcommand{\iid}{\emph{i.i.d. }}

\renewcommand{\epsilon}{\varepsilon}

 \title{\bf Propagation of Chaos for a Balls into Bins Model} 

\author{Nicoletta Cancrini\footnote{nicoletta.cancrini@univaq.it, DIIIE Università dell'Aquila, L'Aquila Italy.}
  \ and Gustavo Posta\footnote{gustavo.posta@uniroma1.it, Dipartimento di Matematica ``G. Castelnuovo'', Sapienza Università di Roma, Roma Italy.}
}

\begin{document}

\noindent

\maketitle
\thispagestyle{empty}

\begin{abstract}
  Consider a finite number of balls initially placed in $L$ bins.
  At each time step a ball is taken from each non-empty bin.
  Then all the balls are uniformly reassigned into bins.
  This finite Markov chain is called Repeated Balls-into-Bins process and is a discrete time interacting particle system with parallel updating.
  We prove that, starting from a suitable (\emph{chaotic}) set of initial states, as $L\to+\infty$, the numbers of balls in each bin becomes independent from the rest of the system \ie we have \emph{propagation of chaos}.
  We furthermore study some equilibrium properties of the limiting \emph{nonlinear process}.
\end{abstract}

\section{Introduction}
\label{sec:intro}

We consider $N$ balls and $L$ bins.
Initially the balls are placed into bins in an arbitrary way.
At each time step a ball is taken from each non-empty bin.
Then all the balls are uniformly reassigned into bins. 

The random evolution of the number of balls in each bin is an ergodic finite state Markov chain called the \emph{Repeated Balls-into-Bins}  (RBB) process \cite{Be:Cl:Na:Pa:Po}.
This system is a conservative interacting particles system in discrete time with parallel updates.
We can think to the RBB process as a zero-range process \cite{Sp} on the complete graph with constant jump rates and parallel updates but, because of the parallel updating, it is not reversible.
For this reason its invariant measure is difficult to compute and, to our knowledge, unknown.

The RBB process appears naturally in different applicative contexts. 
For example we can think to balls in every bin as customers in a queue.
Customers are served at discrete times and each served customer is reassigned to a random queue.
In this setting the RBB process is a discrete time \emph{closed Jackson network} \cite{Ja,Ke}.
The parallel updating is justified \cite{Be:Cl:Na:Pa:Po} by thinking to customers as tasks (or \emph{tokens}) in a network of parallel CPU which are reassigned at every round.

In this paper we are interested in the behavior of the RBB process for large $L$.
We prove that starting from a symmetric initial distribution where the number of balls in each bin becomes independent from the rest of the system as $L\to+\infty$, these properties are preserved for any finite time.
This phenomenon is called \emph{propagation of chaos} \cite{Ka,Sz}.
The limiting evolution of the system is described by a \emph{nonlinear Markov process}.
The interesting fact, citing \cite{Sz}, is that ``the study of every individual gives information on the behavior of the group''.
The price to pay for this simplification is that the limiting process evolves accordingly to a nonlinear equation. 

Propagation of chaos is a largely studied topic in literature, see for example \cite{Sz} and references therein for an introduction.
In recent years, motivated by several applications, the interest focused on models with parallel jumps updating, see for example \cite{An:DP:Fi} and references therein.
Propagation of chaos for these models is not obvious, as parallel jumps may interfere with asymptotic independence.
In particular in \cite{An:DP:Fi} the authors prove propagation of chaos for a wide class of models with simultaneous jumps.
However, due to a different simultaneous jump mechanism, the RBB is not contained in this class and a different approach is needed.

The paper is organized as follows.
In Section~\ref{sec:con} we define the RBB process and the nonlinear process.
Furthermore we study some of the equilibrium properties of the nonlinear process.
In Section~\ref{sec:cp} we prove propagation of chaos for the RBB process.

\section{Construction and basic properties}
\label{sec:con}

In this section we first introduce the RBB process next we define the nonlinear process and the \emph{M/D/1 queue process}.
The nonlinear process is involved in the proof of propagation chaos for the RBB process while the M/D/1 queue process in the study of the equilibrium properties of the nonlinear process.

\subsection{The Repeated Balls-into-Bins process}
\label{sec:bib}

Consider a finite numbers of balls initially placed in an arbitrary way in $L\in\natural$ bins. 
At each time step $t=1,2,\dots$ a ball is taken from each non-empty bin.
Then each ball is uniformly reassigned into bins.
The total number of balls is thus conserved.

More precisely for any $\eta\in\integer_+^L$ define $w(\eta):=(\ind(\eta_1>0),\dots,\ind(\eta_L>0))$ and $L\bar w_L:=\sum_j\ind(\eta_j>0)$.
The RBB process $\eta^L=(\eta^L(t))_{t\geq0}$ has state space $\integer_+^L$ and transition matrix $P_L$ given on functions $f:\integer_+^L\to\complex$ by
\begin{equation*}
  (P_L f)(\eta):=
  \frac{1}{L^{L\bar w_L}}\sum_{\sigma\in\integer_+^L}\binom{L\bar w_L}{\sigma}f(\eta-w+\sigma).
\end{equation*}
So if $\eta^L(t)=\eta$, then
\begin{equation}
  \label{eq:rbb}
  \eta^L(t+1):=\eta-w(\eta)+B(\eta),
\end{equation}
where $B(\eta)$ is an $L$-dimensional multinomial random variable of parameters $L\bar w_L(\eta)$ and $(1/L,\dots,1/L)$.

As the RBB process preserves the number of particles it is clear that it is a finite state Markov chain aperiodic and irreducible so it is ergodic.
Observe that while the transition $(1,\dots,1)\mapsto(L,0,\dots,0)$ is allowed the reverse is not.
This implies that the RBB process is not reversible.
As far as we know an explicit formula for the invariant measure is unknown, as usually happens for non-reversible Markov chains.
Irreversibility comes from the parallel updating mechanism.
For the same model with sequential updating the invariant measure is the uniform (Bose-Einstein) distribution.
For the RBB process this is not the case.
For example if $N=L=3$ the invariant measure puts a mass of $4/21$ on the configuration with one particle in each site, $1/21$ on the 3 configurations with three particles on a single site and $1/9$ on the remaining 6 configurations.

\subsection{The Nonlinear and the M/D/1 queue processes}
\label{sec:nlp}

The nonlinear process $(\eta(t))_{t\in\integer_+}$ takes values in $\integer_+$ and can be defined in the following way.
Let $\eta(0)$ an initial random state, $\rho(0):=\PP(\eta(0)>0)$ and $N_1$ an independent Poisson random variable with expected value $\rho(0)$ (a Poisson random variable with expected value 0 being defined as the 0 random variable).
Then define the new state $\eta(1)$ as:
\begin{equation*}
  \eta(1):=
  \eta(0)-\ind(\eta(0)>0)+N_1.
\end{equation*}
Now iterate the above construction so that, if $\eta(t)$ is the state at time $t>0$, define $\rho(t):=\PP(\eta(t)>0)$ and $N_{t+1}$ an independent Poisson random variable with expected value $\rho(t)$.
Then the new state $\eta(t+1)$ is:
\begin{equation}
  \label{eq:nl1}
  \eta(t+1):=
  \eta(t)-\ind(\eta(t)>0)+N_{t+1}.
\end{equation}

The nonlinear process can be started from any distribution.
However if the initial distribution has finite expectation this property is conserved.

\begin{lemma}
  \label{le:r}
  Assume $\EE(\eta(0))=r\in[0,+\infty]$.
  Then $\EE(\eta(t))=r$, $\forall t\geq0$.
\end{lemma}

\begin{proof}
  The proof is obtained by induction.
  Assume first that $r<+\infty$.
  By equation \eqref{eq:nl1}:
  \begin{equation*}
    \EE\left(\eta(t+1)|\eta(t)\right)
    =\eta(t)-\ind(\eta(t)>0)+\rho(t).
  \end{equation*}
  Then
    \begin{equation*}
      \EE\left(\eta(t+1)\right)=
    \EE\left[\EE\left(\eta(t+1)|\eta(t)\right)\right]=
    \EE(\eta(t))-\PP(\eta(t)>0)+\rho(t)
    =r.
  \end{equation*}
  When $r=+\infty$, again by \eqref{eq:nl1}, we have that $\eta(t+1)$ is obtained by adding a finite mean random variable to a infinite mean one.
\end{proof}

The M/D/1 queue with arrival rate $\rho\geq0$ \cite{Er,Na} is a Markov chain related to the nonlinear process.
It can be defined as follows.
Fix an initial state $\zeta_0\in\integer_+$ and consider a sequence $(M_t)_{t\in\natural}$ of \iid Poisson random variables with expected value $\rho$.
Then $(\zeta(t))_{t\in\integer_+}$ is recursively defined for $t\geq0$:
\begin{equation}
  \label{eq:md1}
  \zeta(t+1):=
  \zeta(t)-\ind(\zeta(t)>0)+M_{t+1}.
\end{equation}

\noindent
Some of the main properties of this chain are given in the next result.
\begin{teo}
  \label{teo:md1}
  The M/D/1 queue with arrival rate $\rho\geq0$ is an irreducible aperiodic Markov chain.
  It is transient for $ \rho>1$, is null persistent for $ \rho=1$ and is positive persistent for $ \rho<1$.
  In the latter case its invariant measure $\pi_\rho$ has characteristic function
  \begin{equation}
    \label{eq:cf1}
      \hat\pi_\rho(x)
      =\frac{(1-\rho)(e^{i x}-1)\exp\left\{\rho (e^{ix}-1)\right\}}{e^{ix}-\exp\left\{\rho (e^{ix}-1)\right\}},
      \qquad\qquad
      x\in\real.
  \end{equation}
  Furthermore if $\rho_r:=1+r-\sqrt{1+r^2}$, $r\geq0$, then $\sum_kk\pi_{\rho_r}(\{k\})=r$.
    \end{teo}

    \begin{proof}
      Equation \eqref{eq:cf1} is equation $(2.3)$ obtained in \cite{Na}.
      The others properties follow by standard Markov chains considerations.
    \end{proof}

The invariant measure of the M/D/1 queue (see equation $(2.4)$ in \cite{Na} for an explicit formula) can be used as a starting distribution of the nonlinear process.
In this case the nonlinear process becomes the M/D/1 queue at equilibrium. 
More precisely we have the following lemma.

\begin{lemma}
  The nonlinear process starting from $\pi_\rho$ is the M/D/1 queue with arrival rate $\rho$.
  Furthermore it is stationary if and only if $\eta(0)\sim\pi_\rho$.
\end{lemma}

\begin{proof}
  Assume that $\eta(0)\sim\pi_{\rho}$.
  For $t=0$ \eqref{eq:nl1} holds with a Poisson random variable $N_{t+1}$ with expected value $\rho$.
  Then equation \eqref{eq:nl1} defines the one step evolution of the M/D/1 queue with arrival rate $\rho$.
  As $\pi_\rho$ is its stationary distribution $\eta(1) \sim\pi_{\rho}$ and $\rho(1)=\PP(\eta(1)>0)=\rho$.
  Iterating this argument we obtain that $\rho(t)=\PP(\eta(t)>0)=\rho$ for any $t\geq0$ and the nonlinear process is the M/D/1 queue with arrival rate $\rho$.
  This proves also that the nonlinear process starting from $\pi_\rho$ is stationary.

  Conversely if the nonlinear process is stationary $\rho(t)=\PP(\eta(t)>0)=\rho(0)$ for any $t\geq0$.
  In this case equation \eqref{eq:nl1} becomes \eqref{eq:md1} which, defines the M/D/1 queue with arrival rate $\rho(0)$.
  As the M/D/1 queue is a Markov chain it is stationary if and only if it is started from equilibrium.
\end{proof}

Next lemma assures a uniform bound on exponential moments of the M/D/1 queue chain and will be used in the proof of Theorem~\ref{teo:st}.
\begin{lemma}
  \label{lem:ui}
    Let $(\zeta(t))_{t\in\integer_+}$ be  the M/D/1 queue with arrival rate $\rho<1$.
    Then there exist positive constants $\lambda_\rho$ and $C_\rho$, depending only on $\rho$, such that for any $\lambda\in(0,\lambda_\rho]$
    \begin{equation*}
      \EE_\zeta(e^{\lambda\zeta(t)})\leq
      C_\rho e^{\lambda\zeta},
    \end{equation*}
    for any $\zeta\in\integer_+$.
\end{lemma}

\begin{proof}
  Let $f(\zeta):=e^{\lambda \zeta}$ and $P$ the transition matrix of $(\zeta(t))_{t\in\integer_+}$.
  We claim that there exist $\gamma\in(0,1)$ and $C>0$, constants depending only on $\rho$, such that
  \begin{equation}
    \label{eq:gc}
    Pf(\zeta)-f(\zeta)\leq
    -\gamma f(\zeta)+C.
  \end{equation}
  Iterating we obtain
  \begin{equation*}
    P^tf(\zeta)\leq
    (1-\gamma)^tf(\zeta)+\frac{C}{\gamma}
    \qquad\qquad
    t\in\integer_+,
  \end{equation*}
  and the result follows.
  To prove \eqref{eq:gc} observe that
  \begin{equation*}
    Pf(\zeta)=
    \EE_\zeta(e^{\lambda\zeta(1)})=
    \exp\{{\lambda(\zeta-\ind(\zeta>0))+\rho(e^\lambda-1)}\}.
  \end{equation*}
  Thus
  \begin{equation*}
   Pf(\zeta)-f(\zeta)=
    \begin{cases}
          \exp\{{\rho(e^\lambda-1)}\}-1 & \text{if $\zeta=0$,}\\
          e^{\lambda\zeta}\big(\exp\{{\rho(e^\lambda-1) -\lambda}\}-1\big) & \text{if $\zeta>0$.}
    \end{cases}
  \end{equation*}
  For any $\rho$ we can find $\lambda_\rho$ such that $\rho(e^\lambda-1) <\lambda$  for any $\lambda\leq\lambda_\rho$.
  Then choosing $C=\exp\{{\rho(e^{\lambda_\rho}-1)}\}$ and $\gamma=1-\exp\{{\rho(e^{\lambda_\rho}-1) -\lambda_\rho}\}$ equation \eqref{eq:gc} follows.
\end{proof}

The following result gives the long time behavior of the distribution of the nonlinear process.
We think that the hypothesis $r<1$ is technical.
It is needed as we prove uniform integrability of the nonlinear process by coupling it with a M/D/1 queue with arrival rate $\rho=r$.

\begin{teo}
  \label{teo:st}
  Assume that $\EE(\eta(0))=r\in[0,1)$ and $\EE(e^{\lambda\eta(0)})<+\infty$ for some $\lambda>0$, then $\eta(t)\Rightarrow\pi_{\rho_r}$ as $t\to+\infty$, where $\rho_r$ has been defined in Theorem~\ref{teo:md1}.
\end{teo}

\begin{proof}
  We first observe that by Lemma~\ref{le:r} we have $\EE(\eta(t))=r$ for any $t\geq0$ and this, via Markov inequality, implies the tightness of the sequence of distributions of $(\eta(t))_{t\in\integer_+}$.
  Furthermore, by equation \eqref{eq:nl1} we have, for any $x\in\real$,
    \begin{equation*}
      \begin{split}
            \EE\left(e^{ix\eta(t+1)}\right)&=
    \EE\left[\EE\left(e^{ix\eta(t+1)}|\eta(t)\right)\right]
    =\EE\left[e^{ix(\eta(t)-\ind(\eta(t)>0))}\EE\left(e^{ix N_{t+1}}|\eta(t)\right)\right]\\
       &=\exp\left\{\rho(t)(e^{ix}-1)\right\} \EE\left(e^{ix(\eta(t)-\ind(\eta(t)>0))}\right).
      \end{split}
  \end{equation*}
  By tightness we can choose a subsequence $(\eta(\bar t))_{\bar t\in\integer_+}$ of $(\eta(t))_{t\in\integer_+}$ with weak limit point $\bar\eta$.
  Taking the limit, for $\bar t\to+\infty$, in the previous equation we get:
  \begin{equation}
    \label{eq:cf}
    \EE\left(e^{ix\bar\eta}\right)=
    \exp\left\{\bar\rho(e^{ix}-1)\right\} \EE\left(e^{ix(\bar\eta-\ind(\bar\eta>0))}\right),
  \end{equation}
  where $\bar\rho=\lim_{\bar t\to+\infty}\rho(t)$.
  Observe that
  \begin{equation*}
    \begin{split}
          \EE\left(e^{ix (\bar\eta-\ind(\bar\eta>0))}\right)&=
    \EE\left(e^{ix \bar\eta},\bar\eta=0\right)+ e^{-ix}\EE\left(e^{ix \bar\eta},\bar\eta>0\right)\\
    &=1-\bar\rho+e^{-ix}\left[\EE\left(e^{ix \bar\eta}\right)-(1-\bar\rho)\right].
    \end{split}
  \end{equation*}
  Plugging this expression in the right hand side of equation \eqref{eq:cf} and solving it we get
  \begin{equation*}
    \EE\left(e^{ix \bar\eta}\right)
    =\frac{(1-\bar\rho)(e^{ix}-1)\exp\left\{\bar\rho (e^{ix}-1)\right\}}{e^{ix}-\exp\left\{\bar\rho (e^{ix}-1)\right\}}.
  \end{equation*}
  Thus, by Theorem~\ref{teo:md1}, $\bar\eta\sim\pi_{\bar\rho}$ and the limit points of the distributions of $(\eta(t))_{t\in\integer_+}$ belong to $\{\pi_\rho:\rho\in[0,1)\}$.
  To identify a unique limit point we will show that
  \begin{equation}
    \label{eq:ui}
    \EE(\bar \eta) =\lim_{\bar t\to+\infty}\EE(\eta(\bar t))=r,
  \end{equation}
  by proving uniform integrability of the sequence $(\eta(t))_{t\in\integer_+}$ (see for example Theorem 25.11 of \cite{Bi}).
  In fact the nonlinear process $(\eta(t))_{t\in\integer_+}$ can be coupled with a M/D/1 queue $(\zeta(t))_{t\in\integer_+}$ with arrival rate $r$, so that $\PP(\eta(t)\leq \zeta(t))=1$ for any $t\geq0$, and uniform integrability of $(\eta(t))_{t\in\integer_+}$ will follow by uniform integrability of $(\zeta(t))_{t\in\integer_+}$.

  First observe that
  \begin{equation*}
    \rho(t)=
    \PP(\eta(t)>0)\leq
    \EE(\eta(t))=r
  \end{equation*}
  and take $\zeta(0)=\eta(0)$.
  For any $t>0$ take a sequence of \iid Bernoulli random variables $Y_{1\, t},Y_{2\, t},\dots$ with parameter $\rho(t)/r$ such that sequences with different $t$ are independent and independent from $M_{t+1}$ in \eqref{eq:md1}.
 Now choose $N_{t+1}$ in \eqref{eq:nl1} as a thinning of $M_{t+1}$:
 \begin{equation*}
   N_{t+1}:=
   \sum_{k=1}^{M_{t+1}}Y_{k\,t}.
 \end{equation*}
 This implies $N_{t+1}\leq M_{t+1}$ \as for any $t\geq0$ so that $\eta(t)\leq\zeta(t)$ \as for any $t>0$.
 To obtain uniform integrability of the nonlinear process observe that by Lemma~\ref{lem:ui},  taking $\lambda>0$ small enough,
 \begin{equation*}
   \EE(e^{\lambda\eta(t)})
   \leq\sum_\eta\PP(\eta(0)=\eta) \EE_\eta(e^{\lambda\zeta(t)})
   \leq C_r \EE(e^{\lambda\eta(0)})<+\infty.
 \end{equation*}
\end{proof}

\section{Propagation of chaos for the RBB process}
\label{sec:cp}

We present here the main result of this paper.
Consider a \emph{chaotic} initial state for the RBB process, \ie an initial state sequence $\eta^L(0)$ with symmetric distribution such that the components become independent as $L\to+\infty$.
We will show that the RBB process preserves this property at each time $t$ and if we look at the time evolution of $\eta^L(t)$ for $t\in[0,T]$, in the limit $L\to+\infty$, it behaves as the product of independent copies of the nonlinear process defined in Section~\ref{sec:nlp}.
This property is known as propagation of chaos for the distribution of $\eta^L$ \cite{Sz}.
More precisely

\begin{teo}
  \label{teo:mai}
  Assume that for any $L\in\natural$ the initial state sequence $\eta^L(0)$ is finitely exchangeable and that for some probability measure $\mu$ on $\integer_+$
  \begin{equation*}
   \lim_{L\to+\infty} \PP(\eta_1^L(0)=\xi_1,\dots,\eta_n^L(0)=\xi_n)=
    \prod_{k=1}^n\mu(\{\xi_k\}),
  \end{equation*}
  for any $n\in\natural$ and $\xi_1,\xi_2,\dots,\xi_n\in\integer_+$.
  
  Then $\left(\eta^L(t)\right)_{t\in\integer_+}$ is finitely exchangeable and for any $T\in\natural$ and for all bounded path functionals $\Phi_{k\,T}$
  \begin{equation*}
    \Phi_{k\,T}(\eta^L):=\Phi_{k\,T}(\eta_k^L(0),\eta_k^L(1),\dots,\eta_k^L(T)),
    \qquad
    k=1,\dots,n
  \end{equation*}
  we have
  \begin{equation}
    \label{eq:mai1}
   \lim_{L\to+\infty} \EE\left[\prod_{k=1}^n\Phi_{k\,T}(\eta^L)\right]=
    \prod_{k=1}^n\EE\left[\Phi_{k\,T}(\eta)\right],
  \end{equation}
  where $\eta=(\eta(t))_{t\in\integer_+}$ is the nonlinear process defined in Section~\ref{sec:nlp} with initial state $\eta(0)\sim\mu$. 
\end{teo}

\begin{proof}
  The exchangeability property follows from the fact that the evolution of the RBB process preserves symmetry at each time step, see equation \eqref{eq:rbb}.
  
  To prove equation \eqref{eq:mai1}, using characteristic functions properties (see for example \cite{Bi} \S 26), it is sufficient to show that
    \begin{equation}
      \label{eq:mai}
   \lim_{L\to+\infty} \EE\left[\exp\left\{i\sum_{s=0}^t\sum_{k=1}^nx_k(s)\eta_k^L(s)\right\}\right]=
    \prod_{k=1}^n\EE\left[\exp\left\{i\sum_{s=0}^tx_k(s)\eta(s)\right\}\right],
  \end{equation}
  for any $t\in\integer_+$ and any $x_k(s)\in\real$, $k=1,\dots,n$, $s=0,\dots,t$.
  We proceed inductively.
  
  For $t=0$ \eqref{eq:mai} holds because by hypotesis $\eta(0)\sim\mu$. 
  We assume \eqref{eq:mai} for some $t$ and we prove it holds for $t+1$.
  Using conditional expectation we have:
  \begin{equation*}
    \begin{split}
         &\EE\left[\exp\Big\{i\sum_{s=0}^{t+1}\sum_{k=1}^nx_k(s)\eta_k^L(s)\Big\}\right]\\
        &=\EE\left[\exp\Big\{i\sum_{s=0}^{t}\sum_{k=1}^nx_k(s)\eta_k^L(s)\Big\}\EE\left(e^{i\sum_{k=1}^nx_k(t+1)\eta_k^L(t+1)}\Big|\,\eta^L(s):s\leq t\right)\right].
    \end{split}
  \end{equation*}
  Define
    \begin{equation*}
    \bar w_L(t)
    :=\frac{1}{L}\sum_{k=1}^L\ind(\eta_k^L(t)>0).
  \end{equation*}
  Using Markov property, equation \eqref{eq:rbb} and the expression of the characteristic function of multinomial distribution we obtain
  \begin{equation}
    \label{eq:fc}
    \begin{split}
           &\EE\left(e^{i\sum_{k=1}^n x_k(t+1)\eta_k^L(t+1)}\Big|\,\eta^L(s):s\leq t\right)\\
           &=e^{i\sum_{k=1}^n x_k(t+1)(\eta_k^L(t)-\ind(\eta_k^L(t)>0))}
\big(1-\frac{1}{L}\sum_{k=1}^n(1-e^{i x_k(t+1)})\big)^{L\bar w_L(t)}
    \end{split}
  \end{equation}
  Define
  \begin{equation}
    \label{eq:psi}
      \Psi_{t}(\eta^L)
      :=\exp\Big\{i\sum_{s=0}^{t}\sum_{k=1}^n x_k(s)\eta_k^L(s)+i\sum_{k=1}^n x_k(t+1)(\eta_k^L(t)-\ind(\eta_k^L(t)>0))\Big\},
  \end{equation}
  \begin{equation*}
    a_L:=\Big(1-\frac{1}{L}\sum_{k=1}^n(1-e^{i x_k(t+1)})\Big)^L,
  \end{equation*}
  and $\rho^L(t):=\EE(\bar w_L(t))$.
  We observe that, from exchangeability and inductive hypothesis at time $t$
  \begin{equation*}
    \lim_{L\to+\infty}\rho^L(t)
    =\rho(t)
    =\PP(\eta(t)>0).
  \end{equation*}
  Furthermore
  \begin{equation*}
        \lim_{L\to+\infty} a_L=\exp\Big\{\sum_{k=1}^n(e^{i x_k(t+1)}-1)\Big\}:=a.
  \end{equation*}
  Then
  \begin{multline*}
        \EE\left[\exp\Big\{i\sum_{s=0}^{t+1}\sum_{k=1}^n x_k(s)\eta_k^L(s)\Big\}\right]=
        \EE\left[\Psi_{t}(\eta^L)a_L^{\bar w_L(t)}\right]\\
        =\EE\left[\Psi_{t}(\eta^L)(a_L^{\bar w_L(t)}-a^{\rho^L(t)})\right]+\EE\left[\Psi_{t}(\eta^L)\right](a^{\rho^L(t)}-a^{\rho(t)})+\EE\left[\Psi_{t}(\eta^L)\right]a^{\rho(t)}.
  \end{multline*}
  The second term above goes to zero as $L\to+\infty$.
  By the inductive hypothesis applied to $\Psi_t$ we obtain for the third term 
  \begin{multline*}
    \lim_{L\to+\infty}\EE\left[\Psi_{t}(\eta^L)\right]a^{\rho(t)}\\
    =\prod_{k=1}^n\EE\Big[\exp\Big\{i\sum_{s=0}^{t} x_k(s)\eta(s)+i x_k(t+1)(\eta(t)-\ind(\eta(t)>0))\Big\}\Big]\\
    \times e^{\rho(t)(e^{i x_k(t+1)}-1)}.
  \end{multline*}
  So the result follows showing that the first term goes to zero as $L\to+\infty$.
  Observe that
  \begin{equation*}
    \Big|\EE\left[\Psi_{t}(\eta^L)(a_L^{\bar w_L(t)}-a^{\rho^L(t)})\right]\Big|
    \leq\EE\left[\big|a_L^{\bar w_L(t)}-a_L^{\rho^L(t)}\big|\right]+\big|a_L^{\rho^L(t)}-a^{\rho^L(t)}\big|.
  \end{equation*}
  The second term above goes to zero as $L\to+\infty$, while for the first one fix $\delta>0$  we have
  \begin{equation*}
    \begin{split}
          \EE\left[\big|a_L^{\bar w_L(t)}-a_L^{\rho^L(t)}\big|\right]&\leq
    \EE\left[\big|a_L^{\bar w_L(t)}-a_L^{\rho^L(t)}\big|,\big|\bar w_L(t) -\rho^L(t)\big|\leq\delta\right]\\
        &+\EE\left[\big|a_L^{\bar w_L(t)}-a_L^{\rho^L(t)}\big|,\big|\bar w_L(t) -\rho^L(t)\big|>\delta\right].
    \end{split}
  \end{equation*}
  By Taylor expansion the first term above can be bounded by a constant, depending only on $x_k(t+1)$, $k=1,\dots,n$, times $\delta$;
  the second term can be bounded by
  \begin{equation*}
    \EE\left[\big|a_L^{\bar w_L(t)}-a_L^{\rho^L(t)}\big|,\big|\bar w_L(t) -\rho^L(t)\big|>\delta\right]\leq
    C \PP(|\bar w_L(t) -\rho^L(t)|>\delta),
  \end{equation*}
  where $C$ is a positive constant depending only on $x_k(t+1)$, $k=1,\dots,n$.
  By using Chebyshev inequality we have that
  \begin{equation}
    \label{eq:ind}
    \PP(|\bar w_L(t) -\rho^L(t)|>\delta)\leq
    \frac{1}{\delta^2}\Big(\frac{1}{4 L}+\PP(\eta_1^L(t)>0,\eta_2^L(t)>0)-\rho^L(t)^2\Big).
  \end{equation}
  Observe that by inductive hypothesis 
  \begin{equation*}
    \lim_{L\to+\infty}\left[\PP(\eta_1^L(t)>0,\eta_2^L(t)>0)-\rho^L(t)^2\right]
    =\PP(\eta(t)>0)^2-\rho(t)^2=0.
  \end{equation*}
  Then, fixed $\epsilon>0$ we can choose $\delta>0$ small enough so that the first term is smaller than $\epsilon/2$ for any $L$ and choosing $L$ large enough the result follows.
\end{proof}

\begin{re}
  The previous result describes the propagation of chaos for the RBB process in finite time intervals.
  Unfortunately this is not enough to prove that, as naturally expected, the same behavior holds in the infinite time limit.
  However consider the RBB process with $L$ bins and $N:=rL$ particles and let $\nu_L^N$ be its stationary measure.
  We can show that if $\nu_L^N$ is chaotic then it is $\pi_{\rho_r}$-chaotic.
  In fact, using \eqref{eq:fc} and stationarity we have
  \begin{equation}
    \label{eq:ms1}
    \EE\Big[e^{i\sum_{k=1}^n x_k\eta_k^L}\Big]
    =    \EE\Big[e^{i\sum_{k=1}^n x_k(\eta_k^L-w_k(\eta^L))}\Big(1-\frac{1}{L}\sum_{k=1}^n(1-e^{ix_k}\Big)^{L\bar w_L}\Big]
  \end{equation}
  for any $x_1,\dots,x_n\in\real$.
   Let $\tilde\eta$ be a weak limit point of $\eta^L$ and, passing to a subsequence, assume that $\eta^L\Rightarrow\tilde\eta$.
  By chaoticity we can follow the same lines of the proof of Theorem~\ref{teo:mai}, take the limit as $L\to+\infty$ in \eqref{eq:ms1} and
    \begin{equation*}
    \EE\Big[e^{i\sum_{k=1}^n x_k\tilde\eta_k}\Big]
    = \EE\Big[e^{i\sum_{k=1}^n x_k(\tilde\eta_k-w_k(\tilde\eta))}\Big]\exp\Big\{\tilde\rho\sum_{k=1}^n(e^{ix_k}-1)\Big\},
  \end{equation*}
  where $\tilde\rho=\PP(\tilde\eta_1>0)$.
  This means that the distribution of $\tilde\eta$ is invariant under the evolution of a product of infinite independent copies of the M/D/1 queue with intensity $\tilde\rho$ and $\tilde\eta\sim\pi_{\tilde\rho}^{\otimes \infty}$.
  To conclude we have to show that $\tilde\rho=\rho_r$.
  
  Differentiating twice \eqref{eq:ms1} and using the fact that $\eta_1^L+\dots+\eta_L^L=rL$, an explicit computation yields
  \begin{equation*}
    \Cov(w_1(\eta^L),w_2(\eta^L))
    =-\EE(w_1(\eta^L))^2+2 \EE(w_1(\eta^L))\frac{(r+1)L-1}{L-1}-2r\frac{L}{L-1}.
  \end{equation*}
  Taking the limit as $L\to+\infty$ in this equation we obtain
    \begin{equation*}
    0=\Cov(w_1(\tilde\eta),w_2(\tilde\eta))
    =-\tilde\rho^2+2\tilde\rho(r+1)-2r,
  \end{equation*}
  which, by Theorem~\ref{teo:md1}, implies $\tilde\rho=\rho_r$.

  Observe that the same lines can be followed if in \eqref{eq:ind} chaoticity of $\nu_L^N$ is replaced by the negative association property $\PP(\eta^L_1>0|\eta^L_2>0)\leq \PP(\eta^L_1>0)$.
\end{re}

\subsection*{Acknowledgements}
We thank Lorenzo Bertini for friendly discussions.
This work has been supported by the PRIN 20155PAWZB ``Large Scale Random Structures''.

\end{document}